\documentclass[12pt,a4paper]{article}
\usepackage[T1]{fontenc}
\usepackage[utf8]{inputenc}
\usepackage[british]{babel}
\usepackage{amsmath,amssymb,amsthm}
\usepackage{tensor}
\usepackage{xcolor}
\usepackage[all]{xy}
\usepackage{cite}
\usepackage{hyperref}

\newcommand{\Z}{\mathbb{Z}}

\DeclareMathOperator{\Syl}{Syl}

\DeclareMathOperator{\Soc}{Soc}
\DeclareMathOperator{\Ker}{Ker}
\DeclareMathOperator{\Aut}{Aut}

\DeclareMathOperator{\Fix}{Fix}

\DeclareMathOperator{\Core}{Core}
\DeclareMathOperator{\id}{id}
\DeclareMathOperator{\SLI}{SLI}

\DeclareMathOperator{\Cent}{C}
\DeclareMathOperator{\Norm}{N}
\DeclareMathOperator{\Ideal}{I}

\newtheorem{teo}{Theorem}[section]
\newtheorem{prop}[teo]{Proposition}

\theoremstyle{definition}
\newtheorem{defi}[teo]{Definition}

\theoremstyle{remark}
\newtheorem{remark}[teo]{Remark}

\theoremstyle{definition}
\newtheorem{ex}[teo]{Example}

\newtheorem{lemma}[teo]{Lemma}

\newtheorem{athm}{Theorem}

\title{Carter-like theorems for skew braces}
\author{Sergio Camp-Mora
\thanks{Departamento de Matemática Aplicada, Universitat Polit\`ecnica de Val\`encia, Camino de Vera, s/n, 46022, Val\`encia, Spain; \texttt{scamp@mat.upv.es}.} \and P. P\'erez-Altarriba
\thanks{Departament de Matem\`atiques, Universitat de Val\`encia, Av. Vicent Andrés Estellés, 19, 46100 Burjassot, Val\`encia, Spain; \texttt{Pedro.A.Perez@uv.es}; ORCID 0009-0009-7082-9002}}
\date{}
\begin{document}
\maketitle

\begin{abstract}
Analogues of the Sylow, Hall, and Schur-Zassenhaus theorems for finite groups have recently been generalised to finite skew braces, thus providing a powerful framework for their structural analysis. Building on this foundation, the natural next step of skew braces is the study of Carter sub-skew braces, that is, centrally nilpotent, self-idealising sub-skew braces. However, while every finite soluble group possesses Carter subgroups, not every finite soluble skew brace possesses Carter sub-skew braces. We introduce Carter sub-skew braces, analyse their structural limitations and establish Carter-like theorems. We prove that finite soluble skew braces possesses Carter sub-skew braces if every sub-skew brace of prime-power order is a centrally nilpotent left ideal. We also prove their existence within finite multipermutational skew braces.


\end{abstract}

\section{Introduction}

Skew left braces are algebraic structures introduced in \cite{GuarnieriVendramin17} as a generalisation of the left braces defined in \cite{Rump07}. They have received considerable attention in recent years, particularly, due to their role in the study of non-degenerate set-theoretic solutions of the Yang-Baxter equation. A \emph{skew (left) brace} is a set $B$ endowed with two group operations, $(B,+)$ and $(B,\cdot)$, satisfying the distributivity-like relation
\[a(b+c)=ab-a+ac \qquad \text{for all } a,b,c\in B.\]
Equivalently, the \emph{multiplicative group} $(B,\cdot)$ acts via automorphisms on the \emph{additive group} $(B,+)$ through the \emph{lambda map}
\[\lambda\colon(B,\cdot)\longrightarrow \Aut(B,+),\qquad \lambda_a(b)=-a+ab.\]
When this action is trivial, $B$ is called a \emph{trivial skew brace}, and the two group operations coincide. Therefore, trivial skew braces can be regarded as groups, showing that skew braces provide a natural generalisation of groups. This makes it natural to investigate group-theoretic notions and properties in the setting of skew braces.

In finite group theory, the structural analysis of soluble groups, which is based on the foundational theorems of Sylow, Hall, Schur–Zassenhaus, and Carter, provides a fundamental framework for understanding groups \cite{DoerkHawkes92}. Sylow and Hall theorems establish the existence, conjugacy and dominance of subgroups of prime-power and coprime orders in finite groups, while Carter theorem establishes the existence and conjugacy of nilpotent self-normalising subgroups in finite soluble groups. These results have played a fundamental role in the development of the theory of finite soluble groups and naturally motivate the development of an analogous theory for skew braces.

A \emph{sub-skew brace} of $B$ is a subgroup of both $(B,+)$ and $(B,\cdot)$. For simplicity, we will simply refer to a sub-skew brace as a \emph{subbrace}. A subbrace $A$ is called a \emph{left ideal} if it is $\lambda$-invariant, a \emph{strong left ideal} if, in addition, $A$ is normal in $(B,+)$, and an \emph{ideal} if it is also normal in $(B,\cdot)$. Ideals are particularly important because they allow us to form quotients of skew braces. 

An important part of the theory of skew braces has developed alongside with the corresponding theory of groups. Building on this line of research given a prime $p$, a \emph{Sylow $p$-subbrace} of $B$ is a Sylow $p$-subgroup of $B$ that is also a subbrace. Similarly, given a set of primes $\pi$, a \emph{Hall $\pi$-subbrace} of $B$ is a Hall $\pi$-subgroup of $B$ that is also a subbrace. In \cite{caranti2026sylowtheoremskewbraces}, Caranti, Del Corso, Di Matteo, Ferrara, and Trombetti proved the first Sylow theorem and various Hall-type theorems for specific classes of skew braces. Subsequently, Truman \cite{truman2026analoguessylowstheoremcauchys} removed the class restrictions and proved unconditional analogues of the first Sylow theorem and Cauchy theorem for all finite skew braces, as well as a Hall-type theorem for skew braces with soluble additive and multiplicative groups. Building upon this work, in \cite{ferrara2026schurzassenhaustheoremsylowstheorem}, Ferrara and Trombetti established a Schur-Zassenhaus theorem for finite skew braces, which guarantees the existence of complements for ideals under coprime conditions (independently proved by Damele in \cite{damele2026schurzassenhaustheoremfiniteskew}). They also proved a Sylow third theorem and recovered the classical containment properties by showing that left ideals are always contained in the corresponding Sylow, Hall, or complement subbraces. Finally, in \cite{ballesterbolinches2026finitetrifactorisedgroupssylow, ballesterbolinches2026finitetrifactorisedgroupssylowdominance}  Ballester-Bolinches, P\'erez-Altarriba, and P\'erez-Calabuig showed that the unconditional results by Truman and the dominance proven by Trombetti and Ferrara, arise naturally from the Sylow and Hall structure of the finite trifactorised groups associated with the skew braces.

With the Sylow, Schur-Zassenhaus, and Hall analogues established, the natural next step is the study of Carter subbraces. A \emph{Carter subbrace} of $B$ is a centrally nilpotent subbrace that is self-idealising, meaning that it is not an ideal of any strictly larger subbrace of $B$. The first remarkable difference  when extending this theory from groups to skew braces is that while every finite soluble group possesses Carter subgroups, not every finite soluble skew brace possesses Carter subbraces. As we show in Example~\ref{ex-non-exis-carter-subbr}, the classical existence theorem fails even for finite soluble skew braces of nilpotent type. The goal of this paper is to develop the theory of Carter subbraces for finite soluble skew braces, understand their structural limitations, and establish hypotheses required to guarantee their existence.

In Section~\ref{sec-basic-prop}, we show that Carter subbraces share some of the classical properties of Carter subgroups; for instance, they are maximal among centrally nilpotent subbraces (Proposition~\ref{prop-carter-max-cn}), but they fail to behave well under quotients (Example~\ref{ex-carter-not-inv-quo}). 

As previously mentioned Carter subbraces may fail to exist for finite soluble skew braces. Therefore, their existence requires additional hypotheses.

\begin{athm}\label{teo-soluble-Carter}
Let $B$ be a finite soluble skew brace such that every subbrace of  prime-power order is a centrally nilpotent left ideal. Then $B$ has Carter subbraces.
\end{athm}

Since in a trivial skew brace every subbrace of prime-power order is a centrally nilpotent left ideal, Theorem~\ref{teo-soluble-Carter} generalises the existence of Carter subgroups in finite soluble groups.

Our second existence theorem concerns finite multipermutational skew braces, which are, in particular, soluble skew braces. This class of skew braces plays a central role in the study of multipermutational solutions to the Yang–Baxter equation.

\begin{athm}\label{teo-multiperm-Carter}
Let $B$ be a finite multipermutational skew brace. Then $B$ has Carter subbraces.
\end{athm}

In this case we can say more. By \cite[Corollary~2.23]{SmoktunowiczVendramin18}, $(B,\cdot)$ is soluble, so Carter subgroups of $(B,\cdot)$ exist. In Section~\ref{sec-Carter-teo-mul} we investigate when such Carter subgroups are subbraces, and in particular Carter subbraces.

\begin{athm}\label{teo-multiperm-Carter-multi}
Let $B$ be a finite multipermutational skew brace. Then there exists a unique Carter subgroup of $(B,\cdot)$ that is a subbrace. Furthermore, it is a Carter subbrace and contains the idealiser of the Sylow system of $B$.
\end{athm}

\section{Preliminaries}

Throughout, all skew braces will be finite. Let $B$ be a skew brace and
$\mathfrak{X}$ a class of groups. We say that $B$ is of
\emph{$\mathfrak{X}$ type} if $(B,+)$ is in $\mathfrak{X}$, and that $B$ is
\emph{bi-$\mathfrak{X}$} if both $(B,+)$ and $(B,\cdot)$ belong to
$\mathfrak{X}$. An \emph{abelian skew brace} is a trivial skew brace of abelian type.

Given subsets $X,Y\subseteq B$, we write $\langle X\rangle_+$ and
$[X,Y]_+$ for the subgroup generated by $X$ and the commutator of $X$
and $Y$ in $(B,+)$, and $\langle X\rangle_\bullet$, $[X,Y]_\bullet$ for
the corresponding objects in $(B,\cdot)$. We also set
\[
X\ast Y=\langle x\ast y=-x+xy-y \mid x\in X,\ y\in Y\rangle_+.
\]
 Given a subset $X\subseteq B$, we define the \emph{subbrace generated} (respectively the \emph{ideal generated}) by $X$, denoted by $\langle X\rangle$ ( respectively $\langle X\rangle^B$), as the smallest subbrace (respectively ideal) of $B$ containing $X$. Given a subbrace $A$, we denote the largest ideal contained in $A$ as $\Core_B(A)$.

The lambda action yields several substructures as the \emph{kernel of lambda} and the \emph{fixed points} of $B$
\begin{align*}
   \Ker\lambda&=\{a\in B \mid \lambda_a(b)=b \text{ for all } b\in B\},\\ 
   \Fix(B)&=\{b\in B \mid \lambda_a(b)=b \text{ for all } a\in B\}.
\end{align*}
Also several ideals as the \emph{socle} and \emph{centre} of $B$
\[
\Soc(B)=\Ker\lambda\cap Z(B,+),
\]
\[
Z(B)=\Soc(B)\cap Z(B,\cdot).
\]
Iterating these constructions gives the \emph{socle series} and the
\emph{centre series}
\[
\Soc_1(B)=\Soc(B),\quad
\Soc_{n+1}(B)/\Soc_n(B)=\Soc\bigl(B/\Soc_n(B)\bigr),
\]
\[
Z_1(B)=Z(B),\quad
Z_{n+1}(B)/Z_n(B)=Z\bigl(B/Z_n(B)\bigr).
\]
We say that $B$ is \emph{centrally nilpotent} if $Z_n(B)=B$ for some
$n\ge1$, and \emph{multipermutational} if $\Soc_n(B)=B$ for some
$n\ge1$. Central nilpotency is the natural skew-brace generalisation of
nilpotency of groups \cite{BonattoJedlicka23,
BallesterEstebanFerraraPerezCTrombetti25-pacificjm-cent-nilpv}, while
multipermutational skew braces play a central role in the study of
multipermutational solutions of the Yang-Baxter equation
\cite[Theorem~4.13]{CedoJespersKubatVanAntwerpenVerwimp23}. If $B$
admits a series of ideals with abelian factors, $B$ is \emph{soluble}
\cite{BallesterEstebanJimenezPerezC24-solubleskewbraces}.

The \emph{Frattini subbrace} of $B$ \cite{BallesterEstebanJimenezPerezC24-solubleskewbraces}
is
\[
\Phi(B)=\bigcap\{M \mid M \text{ is a maximal subbrace of } B\}.
\]
Notice that $\Phi(B)$ coincide with the set of non-subbrace generating elements of $B$.

Given subbraces $A,C$, the \emph{subbrace commutator} defined in
\cite{CampPerezPerez2026} is
\[
[A,C]=\bigl\langle [A,C]_+\cup [A,C]_\bullet\cup A\ast C\cup C\ast A
\bigr\rangle.
\]
We say that $A$ \emph{idealises} (respectively \emph{centralises}) $C$ if $[A,C]\subseteq C$ (respectively $[A,C]=0$). In general, idealisers and centralisers of an arbitrary subbrace need not exist; see \cite[Example~A] {BallesterEstebanFerraraPerezCTrombetti25-pacificjm-cent-nilpv}. However, \cite[Theorem~A]{CampPerezPerez2026} guarantees the existence of the idealiser for left ideals,
\[\Ideal_B(A)=\Norm_{(B,+)}(A,+)\cap \Norm_{(B,\cdot)}(A,\cdot),\]
and of the centraliser for subbraces contained in $\Fix(B)$,
\[
\Cent_B(A)=\Cent_{(B,+)}(A,+)\cap \Cent_{(B,\cdot)}(A,\cdot).
\]

Assume that $B$ is of nilpotent type and let
$\Sigma_B$ be the set of Sylow subgroups of $(B,+)$. By
\cite[Proposition~7.5]{CampPerezPerez2026}, $\Sigma_B$ is also a Sylow
system of $(B,\cdot)$. We call it the \emph{Sylow system} of $B$.
Its \emph{idealiser} is
\[
\Ideal_B(\Sigma_B)
=\bigcap_{P\in\Sigma_B}\Ideal_{B}(P)=\bigcap_{P\in\Sigma_B}\Norm_{(B,\cdot)}(P,\cdot)
=\Norm_{(B,\cdot)}(\Sigma_B,\cdot),
\]
which is a subbrace of $B$ by \cite[Theorem~A]{CampPerezPerez2026}.

\section{Basic properties of Carter subbraces}\label{sec-basic-prop}

\begin{prop}\label{prop-carter-max-cn}
    Let $B$ be a skew brace, then Carter subbraces are maximal centrally nilpotent subbraces.
\end{prop}
\begin{proof}
    Let $C$ be a Carter subbrace and consider $A$ a centrally nilpotent subbrace of $B$ such that $C\subseteq A$. By \cite[Proposition~4.13]{CampPerezPerez2026} if $C\neq A$ there exists a subbrace $S$ of $A$ such that $[C,S]\subseteq C$ and $C\subsetneq S$, arriving to a contradiction. \qedhere
\end{proof}

Carter subbraces are not invariant by quotients.
\begin{ex}\label{ex-carter-not-inv-quo}
Consider {\small \[ (B,+)=K =\langle a, b, x\;|\; 3a = 3b = 2x= 0, [a,b]=[a,x]=0, [x,b]=b\rangle\cong C_3\times S_3.\] }
Let $f$ be the automorphism of  $K$ given by
\[
\begin{aligned}
f\colon a &\longmapsto a\\
b &\longmapsto  b\\
x&\longmapsto b+x
\end{aligned}
\]
Note that $E=\langle f\rangle\cong C_3$.

Consider $G=[K]E$ the semidirect product with respect the natural action, whose elements will be written in multiplicative notation. We will also call $a,b,x,f$ the images to the corresponding elements of $K$ and $E$ by the embedding of $K$ and $E$ in $G$ respectively.

Let $H=\langle af, b, x\rangle\leq G$ by simple computations we conclude that 
\begin{align*}
 H =\{a^ib^jx^kf^i\;|\; 0\leq i,j\leq 2,\; 0\leq k\leq 1\},   
\end{align*}
 $axf$, $b$ and $x$ have order $3$, $3$ and $2$ respectively, $[axf, b]=[axf,x]=1$, $xbx=b^{-1}$ and $H=\langle axf\rangle\times [\langle b\rangle]\langle x\rangle\cong C_3\times S_3$.
From here it is easy to see that $KH=HE=G$ and $H\cap E=1$. Hence, by \cite[Proposition~3.3]{BallesterEstebanPerezAPerezC26-commmathstat-categoriesskewleftbraces} we can define a skew brace structure on $B$, where,
\[(B,\cdot)\cong_\sigma H\cong C_3\times S_3.\]
and $\sigma:H\rightarrow K$, given by $\sigma(ke)=k$, for $k\in K$ and $e\in E$, is a bijective derivation with respect the projection of $H$ on $\Aut(B)$.

For the sake of completeness we include in Table~\ref{table-counter-3imp1} all values of the bijective derivation $\sigma$.

\begin{table}[ht]
    \centering
    {\small \begin{tabular}{cc|cc|cc}
    \hline
        $x$ & $\sigma^{-1}(x)$ & $x$ & $\sigma^{-1}(x)$ &  $x$ & $\sigma^{-1}(x)$ \\\hline
        $0$  & $1$ & $a$ & $af$ & $2a$  & $a^2f^2$ \\
        $b$  & $b$ & $a+b$ & $abf$ & $2a+b$  & $a^2bf^2$ \\        
        $2b$  & $b^2$ & $a+2b$ & $b^2af$ & $2a+2b$  & $a^2b^2f^2$ \\        
        $x$  & $x$ & $a+x$ & $axf$ & $2a+x$  & $a^2xf^2$ \\        
        $b+x$  & $bx$ & $a+b+x$ & $abxf$ & $2a+b+x$  & $a^2bxf^2$ \\        
        $2b+x$  & $b^2x$ & $a+2b+x$ & $ab^2xf$ & $2a+2b+x$  & $a^2b^2xf^2$ \\
        \hline
    \end{tabular}}
    \caption{Bijective derivation of the skew brace $(B,+,\cdot)$.}
    \label{table-counter-3imp1}
\end{table}

This skew brace corresponds to \texttt{S(18, 14)} of the \texttt{YangBaxter} library \cite{VendraminKonovalov25-YangBaxter-0.10.7} of \texttt{GAP} \cite{GAP4-15-1}.

Notice that $\Ker\lambda=\langle b, x\rangle_+$ is a maximal subbrace since its index is prime. And it contain every $2$-subgroup of $(B,+)$. Let us see that $\Ker\lambda$ is the only subbrace of order $6$. The subgroups of $(B,+)$ with order $6$ are
\[\langle a,x\rangle_+,\quad \langle a,b+x\rangle_+,\quad \langle a,2b+x\rangle_+,\text{ and } \Ker\lambda=\langle b,x\rangle_+.\]
Since 
\[a\cdot x=a+b+x,\quad a\cdot (b+x)=a+2b+x,\text{ and } a\cdot (2b+x)=a+x\]
the first three subgroups are not closed under the product, hence, $\Ker\lambda$ is the only subbrace of $B$ with order $6$.

Consider $C$ a $2$-Sylow subbrace of $B$. The only subbraces that contain $C$ are $B$, $\Ker\lambda$ and itself. Since $\Ker\lambda$ is trivial and its additive group is $S_3$, we have that $C$ is not an ideal of $\Ker\lambda$, therefore, $C$ is self-idealising, thus, $C$ is a Carter subbrace of $B$.

Let $I=\langle b\rangle_+$. Notice that $I\subseteq \Ker\lambda\cap \Fix(B)$ and is a normal subgroup of $(B,+)$, therefore, it is an ideal of $B$. Furthermore, $f=\id_K$ modulo $I$ which means that $B/I$ is trivial, in particular, since its order is $9$, is centrally nilpotent.

But $C+I/I\subseteq\Ker\lambda/I$. This means that $C+I/I$ is not a maximal centrally nilpotent subbrace of $B/I$, thus, by Proposition~\ref{prop-carter-max-cn}, it is not a Carter subbrace.
\end{ex}

\section{Carter subbraces in soluble skew braces}\label{sec-Carter-teo}
Carter subbraces could not exists, even for nilpotent type soluble skew braces.
\begin{ex}\label{ex-non-exis-carter-subbr}
Consider $(B,+)=\Z_2\oplus\Z_4$ and define
\begin{equation*}
\begin{bmatrix}
x_1 \\ y_1+2z_1
\end{bmatrix}\cdot \begin{bmatrix}
x_2 \\ y_2+2z_2
\end{bmatrix} = \begin{bmatrix}
x_1+x_2+(x_1+y_1+z_1+x_1y_1)y_2 \\ y_1+2z_1+2y_1x_2+2(x_1+y_1z_1)y_2+y_2+2z_2 \end{bmatrix}. 
\end{equation*}
This gives us a skew brace structure (see \cite[Theorem~3.1]{Bachiller15}), which corresponds to \texttt{S(8, 18)} of the \texttt{YangBaxter} library \cite{VendraminKonovalov25-YangBaxter-0.10.7} of \texttt{GAP} \cite{GAP4-15-1}. 

$B$ is not centrally nilpotent, since $Z(B)=1$ (see \cite[Theorem~3.1]{Bachiller15}). By simple calculations we can see that $A=\langle (1,0),(0,2)\rangle_+$ is the unique maximal subbrace of $B$. It has order $4$, hence, it is centrally nilpotent. Therefore, $A$ is the only candidate for Carter subbrace. But $A$ is an ideal, since $(A,+)$ is characteristic in $(B,+)$ and $(A,\cdot)$ is normal in $(B,\cdot)$. This means that $A$ is not self-idealising. Therefore, $B$ does not have Carter subbraces.

Notice that $(A,+)\cong (A,\cdot )\cong C_2\times C_2$. By \cite[Proposition~2.4]{Bachiller15}, $A$ is an abelian ideal of $B$ and $B/A$ is also abelian. Therefore, $B$ is soluble.
\end{ex}

Therefore, we need to add additional hypotheses. We say that $B$ is $\SLI$ if for every prime $p$ dividing $|B|$ there
exists a Sylow $p$-subbrace of $B$ that is also a left ideal. This class is not closed by subbraces so we need a stronger condition. If every subbrace of $B$ is $\SLI$, we say that $B$ is $\SLI_{subbr}$.

\begin{prop}\label{prop-binilp-satur}
Let $B$ be a finite $\SLI$ skew brace and $I$ an ideal of $B$ contained in $\Phi(B)$ such that $B/I$ is binilpotent. Then $B$ is binilpotent.
\end{prop}
\begin{proof}
Let $P$ be a left ideal Sylow subbrace of $B$. We have that 
\[(P+I)/I\] 
is a Sylow subbrace of $B/I$ so, by hypothesis, $(P+I,+)$ is normal in $(B,+)$. Applying Frattini argument we have that 
\[B=\Norm_{(B,+)}(P,+)+P+I=\Norm_{(B,+)}(P,+)+I.\]
Since $P$ is a left ideal of $B$, by \cite[Proposition~5.3]{CampPerezPerez2026}, $\Norm_{(B,+)}(P,+)$ is a subbrace of $B$. Hence, 
 \[B=\Norm_{(B,+)}(P,+)+I=\Norm_{(B,+)}(P,+).\]
Thus, $P$ is a strong left ideal of $B$. Since this happens for every prime that divides $|B|$ we have that $(B,+)$ is nilpotent.
 
By \cite[Theorem~A]{CampPerezPerez2026}, we have that $\Norm_{(B,\cdot)}(P,\cdot)$ is a subbrace of $B$ so we can repeat the argument on the multiplicative group of $B$ and we arrive to $(B,\cdot)$ being nilpotent.\qedhere
\end{proof}

Proposition~\ref{prop-binilp-satur} is false without the $\SLI$ hypothesis.
\begin{ex}
Following Example~\ref{ex-carter-not-inv-quo}, we have that the $2$-Sylow subbraces of $B$ are 
\[\langle x\rangle_+,\quad \langle b+x\rangle_+,\text{ and } \langle 2b+x\rangle_+.\]
They are not left ideals since $f(x)=b+x$. Therefore, $B$ is not in $\mathcal{SLI}$.

On the other hand, $\Ker\lambda$ is a maximal subbrace and the only subbrace of order $6$. $\Fix(B)=\langle a,b\rangle_+$ is also a maximal subbrace since its index is prime and contains every subbrace of order power of $3$. Therefore, $\Ker\lambda$ and $\Fix(B)$ are the only maximal subbraces of $B$. Hence, $\Phi(B)=\langle b\rangle_+=I$. As we said in Example~\ref{ex-carter-not-inv-quo}, $B/I$ is centrally nilpotent, in particular, binilpotent. But $(B,+)$ is not nilpotent. 
\end{ex}

\begin{lemma}\label{lemma-exist-binilp-self-ideal}
Let $B$ be a finite $\SLI_{subbr}$ skew brace and $I$ an ideal of $B$ such that $I$ and $B/I$ are binilpotent. Then there exists a binilpotent subbrace $C$ such that $B=C+I$. Moreover, if $C$ is maximal among the subbraces with this properties, then it is self-idealising. 
\end{lemma}
\begin{proof}
Suppose that $I\subseteq\Phi(B)$, then, by Proposition~\ref{prop-binilp-satur}, $B$ is binilpotent. So we can suppose that there exists a maximal subbrace $M$ such that $B=M+I$. We have that $I\cap M$ is a binilpotent ideal of $M$ and $M/(I\cap M)\cong B/I$ is binilpotent. By induction we have that there exists a binilpotent subbrace $C$ such that $M=C+(I\cap M)$. Therefore, $B=M+I=C+(I\cap M)+I=C+I$.

Now, suppose that $C$ is maximal among the binilpotent subbraces that supplement $I$ and consider $A$ a subbrace that idealises $C$. By \cite[Proposition~4.12]{CampPerezPerez2026} we can suppose that $C\subseteq A$. Since $B=A+I$ if we prove that $A$ is binilpotent, by the maximality of $C$ we will have that $C$ is self-idealising. Notice that
\[A=A\cap B=A\cap (C+I)=C+(A\cap I)\]
where the last equality follows from Dedekind identity. Since $C$ and $A\cap I$ are binilpotent ideals of $A$ we have that $A$ is binilpotent.\qedhere
\end{proof}

\begin{lemma}\label{lemma-self-idealising-preim}
Let $B$ be a skew brace, $I$ an ideal of $B$ and $D/I$ a self-idealising subbrace of $B/I$. Then $D$ is self-idealising in $B$.
\end{lemma}
\begin{proof}
Let $A$ be a subbrace of $B$ that idealises $D$. Then, by \cite[Proposition~4.3]{CampPerezPerez2026},
\[[(A+I)/I, D/I]=([A,D]+I)/I\subseteq D/I.\]
Since $D/I$ is self-idealising, we have that $(A+I)/I\subseteq D$. Therefore, $A\subseteq D$ and the result follows.\qedhere
\end{proof}

\begin{prop}\label{prop-exist-binilp-self-ideal}
Let $B$ be a finite soluble $\SLI_{subbr}$ skew brace. Then there exists a binilpotent self-idealising subbrace.
\end{prop}
\begin{proof}
Consider $I$ a minimal ideal of $B$ which is abelian, by \cite{BallesterEstebanJimenezPerezC24-solubleskewbraces}, in particular, it is binilpotent. By induction $B/I$ has a binilpotent self-idealising subbrace $D/I$. Since $I$ and $D/I$ are binilpotent, by Lemma~\ref{lemma-exist-binilp-self-ideal}, there exists a binilpotent subbrace $C$ such that $D=C+I$ and is self-idealising in $D$. Let $A$ a subbrace of $B$ that idealises $C$. By \cite[Proposition~4.11]{CampPerezPerez2026}, $A$ idealises $C+I=D$. By Lemma~\ref{lemma-self-idealising-preim}, $D$ is self-idealising in $B$, therefore, $A\subseteq D$. Moreover, $A\subseteq C$ because $C$ is self-idealising in $D$. In conclusion, $C$ is self-idealising in $B$.\qedhere
\end{proof}

\begin{teo}\label{teo-Carter-thechnical}
Let $B$ be a finite soluble $\SLI_{subbr}$ skew brace such that every subbrace of prime-power order is centrally nilpotent. Then $B$ has Carter subbraces.
\end{teo}
\begin{proof}
By Proposition~\ref{prop-exist-binilp-self-ideal}, there exists a binilpotent self-idealising subbrace $C$. Furthermore, by \cite[Theorem~4.12]{BallesterEstebanFerraraPerezCTrombetti25-pacificjm-cent-nilpv}, $C$ is centrally nilpotent. Therefore, $C$ is a Carter subbrace.\qedhere
\end{proof}

\begin{remark}
    Theorem~\ref{teo-Carter-thechnical} is false without the central nilpotency assumption on the subbraces of prime-power order, as we can see in Example~\ref{ex-non-exis-carter-subbr}.
\end{remark}

\begin{proof}[Proof of Theorem~\ref{teo-soluble-Carter}]
$B$ is a finite soluble skew brace where every subbrace of prime-power order is centrally nilpotent. Moreover, given a subbrace, it has Sylow $p$-subbraces, which, by hypothesis, are left ideals. Therefore, we can apply \ref{teo-Carter-thechnical}.\qedhere
\end{proof}

\begin{proof}[Proof of Theorem~\ref{teo-multiperm-Carter}]
A multipermutational skew brace is soluble and of nilpotent type, therefore, every Sylow subbrace is characteristic in $(B,+)$, in particular, lambda invariant. Therefore, $B$ is $\SLI_{subbr}$. Moreover, every subbrace of prime-power order is centrally nilpotent because they are multipermutational. Therefore, we can apply \ref{teo-Carter-thechnical}.\qedhere
\end{proof}

\section{Carter subbraces in multipermutational skew braces}\label{sec-Carter-teo-mul}

Suppose that $G$ is a finite soluble group and consider $H$ a subgroup of $G$ and $\Sigma$ a Sylow system of $G$. We say that $\Sigma$ \emph{reduces} in $H$ if $\{P\cap H\;|\;P\in\Sigma\}$ is a Sylow system of $H$ (see \cite[Kapitel VI, \S 15]{Huppert67}).

\begin{lemma}\label{lemma-sylsys-red-subbr}
Let $B$ be a finite skew brace of nilpotent type and $A$ a subbrace of $B$. Then $\Sigma_B$ reduces in $(A,\cdot)$.
\end{lemma}
\begin{proof}
Since $(B,+)$ is nilpotent, $\Sigma'=\{P\cap A\;|\;P\in\Sigma_B\}$ is the unique a Sylow system of $(A,+)$, which means that $\Sigma_A=\Sigma'$. By \cite[Proposition~7.5]{CampPerezPerez2026}, $\Sigma'$ is a Sylow system of $(A,\cdot)$, therefore, $\Sigma_B$ reduces in $(A,\cdot)$.\qedhere
\end{proof}

\begin{prop}\label{prop-nilp-type-Carter-mult-subbr}
Let $B$ be a finite skew brace of nilpotent type. Then there exists at most one Carter subgroup of $(B,\cdot)$ that is a subbrace, namely the unique Carter subgroup of $(B,\cdot)$ in which $\Sigma_B$ reduces.
\end{prop}

\begin{proof}
By Lemma~\ref{lemma-sylsys-red-subbr}, the Sylow system of $B$ reduces in every Carter subgroup of $(B,\cdot)$ that is a subbrace, but, by {\cite[Theorem~6]{ALPERIN1964355}}, $\Sigma_B$ reduces exactly in only one Carter of $(B,\cdot)$. Hence, there is at most one Carter subgroup of $(B,\cdot)$ such that is a subbrace.\qedhere
\end{proof}

Let $B$ be a finite skew brace of nilpotent type. We will denote by $(\Cent_{\Sigma_B},\cdot)$ to the unique Carter subgroup of $(B,\cdot)$ in which $\Sigma_B$ reduces.

\begin{prop}\label{prop-sys-id-contained}
    Let $B$ be a finite skew brace of nilpotent type. Then $\Ideal_B(\Sigma_B)\subseteq \Cent_{\Sigma_B}$.
\end{prop}
\begin{proof}
    Since $\Sigma_B$ is a Sylow system of $(B,\cdot)$ that reduces in $\Cent_{\Sigma_B}$, the result follows from \cite[Kapitel VI, Satz~15.6]{Huppert67}.\qedhere
\end{proof}

\begin{prop}\label{prop-CN-subbr-Carter}
Let $B$ be a finite skew brace of nilpotent type such that every Sylow subbrace is centrally nilpotent. If $\Cent_{\Sigma_B}$ is a subbrace, then $\Cent_{\Sigma_B}$ is a Carter subbrace of $B$. 
\end{prop}
\begin{proof}
Since $\Cent_{\Sigma_B}$ is  binilpotent and every Sylow subbrace of $\Cent_{\Sigma_B}$ is centrally nilpotent, by \cite[Theorem~4.12]{BallesterEstebanFerraraPerezCTrombetti25-pacificjm-cent-nilpv}, $\Cent_{\Sigma_B}$ is centrally nilpotent. Now, let $A$ be a subbrace such that $[A,\Cent_{\Sigma_B}]\subseteq \Cent_{\Sigma_B}$, then $[A,\Cent_{\Sigma_B}]_\bullet\subseteq \Cent_{\Sigma_B}$. Hence, $(A,\cdot)\leq \Norm_{(B,\cdot)}(\Cent_{\Sigma_B},\cdot)=(\Cent_{\Sigma_B},\cdot)$ and $\Cent_{\Sigma_B}$ is self-idealising.\qedhere
\end{proof}

Without the hypothesis on the Sylow subbraces Proposition~\ref{prop-CN-subbr-Carter} is false. Consider $B$ as in Example~\ref{ex-non-exis-carter-subbr}. We have that $\Cent_{\Sigma_B}=B$ since $(B,\cdot)$ is nilpotent. Therefore, $\Cent_{\Sigma_B}$ is a subbrace but is not centrally nilpotent. 

\begin{lemma}\label{lemma-mult-carter-quo}
    Let $B$ be a finite skew brace of nilpotent type and $I$ an ideal of $B$. Then $\Cent_{\Sigma_B}I/I=\Cent_{\Sigma_{B/I}}$.
\end{lemma}
\begin{proof}
    If $\Sigma_B=\{P_1,\ldots,P_n\}$, then $\Sigma_{B/I}=\{P_1I/I,\ldots, P_nI/I\}$. We have that $(\Cent_{\Sigma_B}I/I,\cdot)$ is a Carter subgroup of $(B/I,\cdot)$. Since $\Sigma_B$ reduces in $\Cent_{\Sigma_B}$, we have that $\Sigma_{B/I}$ reduces in $\Cent_{\Sigma_B}I/I$. By the unicity of $\Cent_{\Sigma_{B/I}}$, the result follows.\qedhere
\end{proof}

\begin{lemma}\label{lemma-cont-sylow-ideliser}
    Let $B$ be a finite skew brace of nilpotent type, $(A,\cdot)$ a nilpotent subgroup of $(B,\cdot)$ in where $\Sigma_B$ reduces and $I$ an ideal of $B$ that is a $p$-brace for some prime $p$. If $B=AI$, then $A\subseteq \Ideal_B(\Sigma_B)$.
\end{lemma}
\begin{proof}
    Consider $q\neq p$ and let $Q$ be the Sylow $q$-subgroup in ${\Sigma}_B$. Since ${\Sigma}_B$ reduces in $(A,\cdot)$, $Q\cap A$ is a Sylow $q$-subgroup of $(A,\cdot)$. $B=A+I$ and $I$ is a $p$-brace, therefore, the Sylow $q$-subgroups of $(A,\cdot)$ and $(B,\cdot)$ have the same order. Thus, $Q=Q\cap A$ which is normal in $(A,\cdot)$. 
    
    On the other hand, if $P$ is the Sylow $p$-subgroup in ${\Sigma}_B$, then $P=P_1 I$, for some $P_1 \in \Syl_p(A)$. So $A\subseteq \Norm_{(B,\cdot)}(P)$ because $A\subseteq \Norm_{(B,\cdot)}(P_1)$ and $(I,\cdot)$ is normal in $(B,\cdot)$. Therefore, $A\subseteq \Norm_{(B,\cdot)}(\Sigma_B))=\Ideal_B(\Sigma_B)$.\qedhere
\end{proof}

\begin{defi}
Let $B$ a finite skew brace. We will say that $B$ is separable if every chief factor of $B$ is a $p$-brace, for some prime $p$.
\end{defi}

\begin{prop}\label{prop-sep-Carter-subbr}
    Let $B$ be a finite separable skew brace of nilpotent type. Then $\Cent_{\Sigma_B}$ is a subbrace of $B$.
\end{prop}

\begin{proof}
We proceed by induction on the cardinality of $B$. Let $I$ be a minimal ideal of $B$, which is a $p$-brace for some prime $p$ by definition. BY Lemma~\ref{lemma-mult-carter-quo}, $\Cent_{\Sigma_B}I/I=\Cent_{\Sigma_{B/I}}$. By induction, $\Cent_{\Sigma_B}I/I$ is a subbrace of $B/I$. Therefore, $A=\Cent_{\Sigma_B}I$ is a subbrace. By Lemma~\ref{lemma-sylsys-red-subbr}, $\Sigma_B$ reduces in it, which implies that $\Cent_{\Sigma_B}=\Cent_{\Sigma_A}$. If $A$ is a proper subbrace of $B$, we may use again the induction hypothesis and conclude that $\Cent_{\Sigma_B}$ is a subbrace of $B$. So we can assume $B=\Cent_{\Sigma_B}I$. By Lemma~\ref{lemma-cont-sylow-ideliser} and Proposition~\ref{prop-sys-id-contained}
\[\Cent_{\Sigma_B}\subseteq \Ideal_B(\Sigma_B)\subseteq \Cent_{\Sigma_B}.\]
Which means that $\Cent_{\Sigma_B}$ coincide with $\Ideal_B({\Sigma}_B)$ which is a a subbrace.\qedhere
\end{proof}


 \begin{proof}[Proof of Theorem~\ref{teo-multiperm-Carter-multi}]
By \cite{pSolubility2026}, $B$ is separable of nilpotent type and every Sylow subbrace is centrally nilpotent. By Proposition~\ref{prop-nilp-type-Carter-mult-subbr} and Proposition~\ref{prop-sep-Carter-subbr} $\Cent_{\Sigma_B}$ is the unique Carter subgroup of $(B,\cdot)$ that is a subbrace. By Proposition~\ref{prop-CN-subbr-Carter} we have that $\Cent_{\Sigma_B}$ is a Carter subbrace of $B$. Finally $\Ideal_B(\Sigma_B)\subseteq \Cent_{\Sigma_B}$, by Proposition~\ref{prop-sys-id-contained}.\qedhere
 \end{proof}

\section*{Declarations}
\subsection*{Author contributions} All authors contributed equally to this work.

\subsection*{Data Availability Statement} No datasets were generated or analysed during
the current study.

\subsection*{Acknowledgments}
We would like to thank the referees for their useful comments and remarks that really improved the paper.

\subsection*{Funding} 
S. Camp-Mora  did not receive funding from any organization for the submitted work.

P. Pérez-Altarriba has been supported by the grant FPU24/01319, funded by the Ministerio de Ciencia, Innovaci\'on y Universidades for the formation of university teachers, by the grant PID2024-159495NB-I00, funded by MICIU/AEI/10.13039/501100011033 and by ERDF/EU and by the grant CIAICO/2023/007 from the
Conselleria d’Educació, Cultura, Universitats i Ocupació, of the Generalitat Valenciana.

\subsection*{Conflict of interest} The authors declare no conflict of interest.

\bibliographystyle{plain}
\bibliography{bibgroup}
\end{document}